\documentclass[12pt]{amsart}
\usepackage{amsmath,amssymb}

\newtheorem{theorem}{Theorem}[section]
\newtheorem{corollary}[theorem]{Corollary}
\newtheorem{lemma}[theorem]{Lemma}
\newtheorem{proposition}[theorem]{Proposition}

\theoremstyle{definition}

\newtheorem{example}[theorem]{Example}

\theoremstyle{remark}

\newcommand{\C}{\mathbb C}
\def \bC{\mathbb C}

\newcommand{\im}{\mathrm{Im}}
\newcommand{\re}{\mathrm{Re}}

\DeclareMathOperator{\jac}{Jac}
\DeclareMathOperator{\rk}{rk}

\begin{document}

\title[Transversality of holomorphic mappings]{Transversality of holomorphic mappings between real hypersurfaces in complex spaces of different dimensions}
\author{Peter Ebenfelt}
\address{Department of Mathematics, University of California at San Diego, La Jolla, CA 92093-0112}
\email{pebenfel@math.ucsd.edu,}
\author{Duong Ngoc Son}
\address{Department of Mathematics, University of California at San Diego, La Jolla, CA 92093-0112}
\email{snduong@math.ucsd.edu}
\thanks{The first author was partly supporting by the NSF grant DMS-1001322. The second author is partially supported by a scholarship from the Vietnam Education Foundation.}
\begin{abstract} We consider holomorphic mappings $H$ between a smooth real hypersurface $M\subset \bC^{n+1}$ and another $M'\subset \bC^{N+1}$ with $N\geq n$. We provide conditions guaranteeing that $H$ is transversal to $M'$ along all of $M$. In the strictly pseudoconvex case, this is well known and follows from the classical Hopf boundary lemma. In the equidimensional case ($N=n$), transversality holds for maps of full generic rank provided that the source is of finite type in view of recent results by the authors (see also a previous paper by the first author and L. Rothschild). In the positive codimensional case ($N>n$), the situation is more delicate as examples readily show. In recent work by S. Baouendi,  the first author, and L. Rothschild, conditions were given guaranteeing that the map $H$ is transversal outside a proper subvariety of $M$, and examples were given showing that transversality may fail at certain points.

One of the results in this paper implies that if $N\le 2n-2$, $M'$ is Levi-nondegenerate, and $H$ has maximal rank outside a complex subvariety of codimension $2$, then $H$ is transversal to $M'$ at all points of $M$. We show by examples that this conclusion fails in general if $N\geq 2n$, or if the set $W_H$ of points where $H$ is not of maximal rank has codimension one. We also show that $H$ is transversal at all points if $H$ is assumed to be a finite map (which allows $W_H$ to have codimension one) and the stronger inequality $N\leq 2n-3$ holds, provided that $M$ is of finite type.
\end{abstract}
\maketitle

\section{Introduction}

The geometric property of transversality is often a crucial ingredient in proving results of an analytic nature in CR geometry and other areas of analysis. We mention here as a general example the frequent use of the Hopf boundary point lemma (a transversality result) in elliptic PDE and potential theory. The reader is referred to e.g.\ \cite{ER06} and \cite{BER07} for a more detailed account of the significance of transversality in CR geometry. In this paper, we shall prove transversality results concerning mappings in CR geometry.

Let $M \subset \C^{n+d}$ and $M' \subset \C^{N+D}$ be smooth generic (in particular CR) submanifolds of codimension $d$ and $D$, respectively (so that the CR dimensions are $n$ and $N$, respectively), and $H$ a holomorphic mapping from an open neighborhood $U$ of $M$ in $\bC^{n+d}$ into $\C^{N+D}$ such that $H(M) \subset M'$. The study of how the CR geometries of $M$ and $M'$ influence the geometric behaviour of $H$, in particular its CR transversality to $M$, has received considerable attention from many authors over the years (see e.g.\ \cite{P77}, \cite{Forn78}, \cite{BRgeom}, \cite{BRhopf}, \cite{BHR95}, \cite{CR94}, \cite{CR98}, \cite{ER06}, \cite{HZ09}, \cite{LM06}, \cite{Zha}, \cite{BH}, \cite{ES10} and the references therein). The equidimensional case ($N=n$ and $D=d$) is by now well understood (in the finite type case) \cite{ES10}: {\it Assume that $H$ has full generic rank (i.e.\ $\jac H:=\det H_Z\not\equiv 0$), $p\in M$, and $M$ is of finite type at $p$. Then, $H$ is CR transversal to $M'$ at $p$.} (This result under the stronger assumption that $H$ is assumed to be a finite map at $p$ was proved earlier in \cite{ER06}.) If $M$ is assumed to be holomorphically nondegenerate, then the condition $\jac H\not\equiv 0$ is also necessary in this context. In this paper, we shall consider the more delicate situation where the CR dimension $N$ of the target $M'$ is larger than that, $n$, of the source $M$. We shall restrict to the case where $M$ and $M'$ are hypersurfaces ($d=D=1$). In this case, the notion of CR transversality coincides with the (in general) weaker notion of transversality in the traditional sense (see \cite{ER06}); recall that a holomorphic mapping $H$ from an open set $U\subset \bC^{n+1}\to \bC^{N+1}$ is said to be transversal (see, e.g., \cite{GG86}) to a real hypersurface $M'\subset \bC^{N+1}$ at $p\in U$ if $p':=H(p)\in M'$ and
\begin{equation}\label{transdef}
T_{H(p)} M' + dH(T_p\, \C^{n+1}) = T_{H(p)}\C^{N+1}.
\end{equation}
In \cite{BER07}, the case of real-analytic hypersurfaces $M\subset\bC^{n+1}$ and $M'\subset \bC^{N+1}$ was considered under the (obviously necessary condition) that the map $H\colon U\subset \bC^{n+1}\to \bC^{N+1}$ sending $M$ into $M'$ does not collapse all of $U$ into $M'$. Sufficient conditions involving the signature and rank of the Levi form $\mathcal L'$ of $M'$, and the CR dimensions $n$ and $N$ were given guaranteeing that the map $H$ must be transversal {\it outside a proper, real-analytic subvariety of} $M$, but not necessarily at a specific point $p\in M$. Examples and results were also given showing that these results are essentially sharp in the sense that if the conditions are violated, then the map $H$ could be non-transversal over all of $M$, but also that under the conditions given transversality can fail at certain points. In this paper, we shall give more restrictive conditions on the rank $r$ of $\mathcal L'$, the CR dimensions $n$ and $N$, and on the map $H$ that will guarantee transversality at {\it all} points.

To formulate our main results, we shall need to introduce a little more notation. Given a holomorphic map $H\colon U\subset \bC^{n+1}\to\bC^{N+1}$, we shall consider the complex analytic subvariety
\begin{equation}\label{WH}
W_H:=\{Z\in U\colon \rk H_Z(Z)<n+1\},
\end{equation}
where $H_Z$ denotes the $(N+1)\times (n+1)$ matrix of partial derivatives of the components of $H$,
$$
H_Z:=\left (\frac{\partial H_i}{\partial Z_j}\right ),\quad 1\leq i\leq N+1,\ 1\leq j\leq n+1.
$$
We shall only consider situations where $W_H$ is a proper subvariety (i.e.\ the rank of $H$ is of generic maximal rank); just as in the equidimensional case mentioned above, this is essentially necessary for transversality to hold under some mild conditions on $M$. Observe that if $\delta_l(Z)$, for $l=1,\ldots,m\leq\binom{N+1}{N-n}$, denote the collection of all non-trivial $(n+1)\times (n+1)$-minors of the matrix $H_Z(Z)$, then $W_H$ coincides with the set defined by
$$\delta_1(Z)=\ldots=\delta_m(Z)=0.$$
Thus, when $N>n$ the codimension of this set is in general large, and the codimension is one only when all the minors have a common divisor.  Our first result is the following:

\begin{theorem}\label{main} Let $M\subset \C^{n+1}$ and $M'\subset \C^{N+1}$ be smooth real hypersurfaces
through $p$ and $p'$ respectively, and $H: (\C^{n+1},p) \to  (\C^{N+1},p')$ a germ at $p$ of holomorphic
mapping such that $H(M)\subset M'$. Denote by $r$ the rank of the Levi form of $M'$ at $p'$ and assume that
\begin{equation}\label{codim}
 2N-r \le 2n- 2.
\end{equation}
If the germ at $p$ of the analytic variety $W_H$, given by \eqref{WH}, has codimension at least $2$, then $H$ is transversal to $M'$ at $p$.
\end{theorem}

The following example shows that condition \eqref{codim} in Theorem \ref{main} is at least `almost'' sharp.
\begin{example}\label{ex0.5} Consider the strictly pseudoconvex hyperquadric $M\subset\bC^{n+1}$ (biholomorphically equivalent to the sphere) given by
$$
\im\, w-\sum_{j=1}^{n} |z_j|^2=0
$$
and the nondegenerate hyperquadric $M'\subset \bC^{2n+1}$ given by
$$
\im\, w'+\sum_{j=1}^{n}|z'_{2j-1}|^2 -\sum_{j=1}^n |z'_{2j}|^2=0,
$$
where we use coordinates $(z,w)\in \bC^n\times \bC$ and $(z',w,)\in \bC^{2n}\times \bC$.
Now, consider the polynomial mapping $H=(F_1,F_2,\ldots,F_{2n-1},F_{2n},G)\colon (\bC^{n+1},0)\to (\bC^{2n+1},0)$ given by
\begin{multline}
H(z,w):=\bigg(z_1+[z]\,z_1+\frac{i}{2}w,z_1-[z]\,z_1-\frac{i}{2}w,
\ldots,\\ z_n+[z]\,z_n+\frac{i}{2}w,z_n-[z]\,z_n-\frac{i}{2}w, -2[z]w \bigg),
\end{multline}
where we have used the notation $[z]:=\sum_{j=1}^nz_j$. We claim that $H$ sends $M$ into $M'$, $H$ is a local embedding at $0$ (and hence as germs at $0$, we have $W_H=\emptyset$), but $H$ is not transversal to $M'$ along the intersection of $M$ and the real hypersurface $\re\, [z]=0$, and hence, in particular, is not transversal at $0$ . For the reader's convenience, a proof of this claim is given in Section \ref{example}. In this example, $N=2n$ and $r=N=2n$ (since $M'$ is Levi nondegenerate). Thus, we have  $2N-r=N=2n$, which is equal to $(2n-2)+2$ and hence condition \eqref{codim} is violated. However, the authors do not know of an example where $2N-r=(2n-2)+1=2n-1$, which leaves open the possibility that condition \eqref{codim} could be sharpened to $2N-r\leq 2n-1$ in Theorem \ref{main}.
\end{example}

We would like to point out that when the target $M'$ is Levi nondegenerate at $p'$ (i.e.\ $r=N$, as in Example \ref{ex0.5} above), then the condition \eqref{codim} can be rewritten $N-n\leq n-2$. (The number $N-n$, the difference between dimension of the target space and the source space, is often referred to as the {\it codimension} of the map.) That is, transversality holds at $p$ for maps $H$ up to a codimensional gap $N-n$ that increases with the CR dimension $n$ of the source manifold, {\it provided} that the codimension of $W_H$ is at least 2. The following example shows that this phenomenon fails if we allow $W_H$ to have codimension one.

\begin{example} \label{ex1}  Consider the sphere $M\subset\bC^{n+1}$ given by
$$
\sum_{j=1}^{n+1} |Z_j|^2-1=0
$$
and the nondegenerate hyperquadric $M'\subset \bC^{n+3}$ given by
$$
\im\, w'-\left(\sum_{j=1}^{n+1}|z'_j|^2 - |z'_{n+2}|^2\right)=0.
$$
It is straightforward to verify that the polynomial mapping $H\colon \bC^{n+1}\to\bC^{n+3}$ given by
$$
H(Z):=(Z_1^2,Z_1Z_2,\ldots Z_1Z_{n+1}, Z_1,0)
$$
sends $M$ into $M'$. The set $W_H$ is given by $Z_1=0$ (and hence has codimension one), and the mapping $H$ is not transversal to $M'$ along the intersection of the sphere $M$ with $W_H$ (cf. Example 2.3 in \cite{BER07}). Thus, this is a family of examples where $W_H$ has codimension one, the map has codimension 2 (i.e.\ $N-n=2$), and transversality fails at certain points regardless of the CR dimension $n$ of the source.
\end{example}

Example \ref{ex1} shows that even for Levi nondegenerate hypersurfaces and maps of generic full rank, transversality may fail at specific points unless further conditions are imposed. One direction is to assume conditions relating the signatures of the Levi forms as in \cite{BH} in which transversality is proved for maps between hyperquadrics of the same signature. We shall not pursue this direction in the present paper.

We note that in Example \ref {ex1} the map $H$ sends the whole hyperplane $W_H=\{Z\colon Z_1=0\}$ to $0\in M'$. In particular, $H$ is not a finite map at $0\in \bC^{n+1}$. Recall that a (germ at $p$ of a) map $H\colon (\bC^{n-1},p)\to (\bC^{N+1},p')$ is finite if $H^{-1}(p')=\{p\}$ as germs at $p$, or equivalently if the vector space $\bC[[Z]]/I(H_1,\ldots,H_{N+1})$ is finite dimensional over $\bC$; here, $\bC[[Z]]$ denotes the ring of formal power series in $Z$ and $I(H_1,\ldots,H_{N+1})$ denotes the ideal generated by the components of $H$. Our next results asserts that if we sharpen condition \eqref{codim} slightly and require $M$ to be of finite type, then transversality holds at all points for finite maps.

\begin{theorem}\label{finitemap} Let $M\subset \C^{n+1}$ and $M'\subset \C^{N+1}$ be smooth real hypersurfaces
through $p$ and $p'$ respectively, and $H: (\C^{n+1},p) \to  (\C^{N+1},p')$ a germ at $p$ of holomorphic
mapping such that $H(M)\subset M'$. Denote by $r$ the rank of the Levi form of $M'$ at $p'$ and assume that
\begin{equation}\label{codim2}
 2N-r \le 2n- 3.
\end{equation}
Assume also that $M$ is of finite type at $p$ and $H$ is a finite map at $p$. Then, $H$ is transversal to $M'$ at $p$.
\end{theorem}

We would like to point out that the map $H$ in Example \ref{ex0.5} is a finite map (indeed, it is a local embedding at $p=0$ and therefore locally 1-to-1 there), showing that condition \eqref{codim2} cannot be improved by much, and the rate of growth of the codimensional gap where transversality holds grows like $n$. We would also like to point out the well known fact that if $M$ is Levi nondegenerate at $p$, then transversality of $H$ at $p$ implies that $H$ is in fact a transversal local embedding at $p$.

Theorem \ref{finitemap} will be a direct consequence of a more general result, which we will now present. Let $s$ be an integer with $1\leq s\leq n+1$, and define
\begin{equation}\label{WHs}
W_H^s:=\{Z\in \bC^{n+1}\colon \rk H_Z(Z)<s\}.
\end{equation}
We note that $W_H^{n+1}=W_H$. Each $W_H^s$ is a complex analytic variety defined by the vanishing of all $k\times k$ minors of $H_Z$, for $k=s,\ldots, n+1$, and we have a nesting
$$
W_H^1\subset W_H^1\subset\ldots\subset W_H^{n+1}=W_H.
$$
Our next result is the following.

\begin{theorem}\label{mains} Let $M\subset \C^{n+1}$ and $M'\subset \C^{N+1}$ be smooth real hypersurfaces
through $p$ and $p'$ respectively, and $H: (\C^{n+1},p) \to  (\C^{N+1},p')$ a germ at $p$ of holomorphic
mapping such that $H(M)\subset M'$. Denote by $r$ the rank of the Levi form of $M'$ at $p'$. Assume that $M$ is of finite type at $p$ and that, for some $1\leq s\leq n+1$,
\begin{equation}\label{codims}
 2N-r \le n+s-3.
\end{equation}
If the germ at $p$ of the analytic variety $W_H$, given by \eqref{WH}, is proper (i.e.\ $H$ has generic rank $n+1$) and the germ at $p$ of $W_H^s$, given by \eqref{WHs}, has codimension at least $2$, then $H$ is transversal to $M'$ at $p$.
\end{theorem}

We note that if $H$ is a finite map at $p$, then $W_H$ is proper and $W_H^n$ has codimension at least 2. Thus, Theorem \ref{finitemap} follows from Theorem \ref{mains} with $s=n$.

We end this introduction by pointing out that if $H$ is a smooth CR mapping from $M$ to $M'$, then we can identify $H$ with a formal holomorphic power series mapping in the variable $Z\in \bC^{n+1}$ centered at $Z=p$ and sending $M$ into $M'$ (formally); see e.g.\ \cite{BER99a} and \cite{BER99b}. The definition of transversality \eqref{transdef} makes sense for formal mappings, and Theorems \ref{main}, \ref{finitemap}, and \ref{mains} remain true for such maps, provided that the algebraic definition of finite map is used in Theorem \ref{finitemap}, and the conditions on the analytic varieties $W_H^s$ in Theorems \ref{main} and \ref{mains} are interpreted algebraically as in Theorems \ref{funny} and \ref{boring}.

\section{Preliminaries}
In this section we will summarize some basic facts and definitions that will be used in this paper. We refer the reader to the book \cite{BER99a} for more details.

A smooth real hypersurface in $\C^{n+1} (\cong \mathbb R^{2n+2})$ is a subset $M$, locally defined by the vanishing of a local defining equation $\rho(Z,\bar Z) = 0$, where $\rho$ is a smooth real-valued function satisfying $d\rho \ne 0$ along $M$. Such $M$ is a CR manifold with CR bundle $T^{(0,1)}M$ whose fiber at $p\in$ is defined by $ T_p^{(0,1)} M : = \C T_pM \cap T^{0,1} \C^{n+1}$.  Sections of $T^{(0,1)}M$ are called CR vector fields. A real hypersurface $M$ is said to be of finite type at $p$ (in the sense of Kohn and Bloom-Graham) if the (complex) Lie algebra
$\mathfrak g_M$ generated by all CR vector fields and their conjugates near $p$ satisfies $\mathfrak g_M(p) = \C T_pM$. The complex conjugate bundle $\overline{T^{(0,1)}}M$ is denoted by $T^{(1,0)}M$. Associated to $M$ at $p$, there is a Hermitian form $\mathcal L_p\colon T_p^{(1,0)}M\times T_p^{(1,0)}M \to \bC T_pM/(T_p^{(1,0)}M + T_p^{(0,1)}M)\cong \bC$ called the Levi form of $M$ at $p$. In terms of a local defining equation $\rho=0$, the space $T^{(1,0)}_pM$ can be identified with the subspace of $c\in \bC^{n+1}$ such that
$$
\sum_{j=1}^{n+1}\frac{\partial\rho}{\partial Z_j}(p,\bar p)c_j=0,
$$
and then the Levi form $\mathcal L_p$ is represented by the restriction to this space of the Hermitian $(n+1)\times(n+1)$-matrix
$$
\rho_{Z\bar Z}(p,\bar p):=\left(\frac{\partial^2\rho}{\partial Z_i\partial \bar Z_j}(p,\bar p)\right),\quad 1\leq i,j\leq n+1.
$$
If $U$ is an open neighborhood of $M$ in $\bC^{n+1}$ and $H\colon U\colon \bC^{N+1}$  a holomorphic mapping, then $H$ sends $M$ into a smooth real hypersurface $M'\subset \bC^{N+1}$ if and only if there is a smooth function $a$ in $U\subset \bC^{n+1}$ such that
$\rho'\circ H=a\rho$, where $\rho'$ denotes a defining function for $M'$. Moreover, $H$ is transversal to $M'$ precisely at those points $p\in M$ where $a\neq 0$ (see e.g.\ \cite{ER06}). In what follows, we shall always assume, without loss of generality of course, that the given points $p\in M$ and $p'=H(p)\in M'$ are both the origin $p=0\in \bC^{n+1}$, and $p'=0\in \bC^{N+1}$.

When $M$ and $M'$ are real-analytic, then $\rho$ and $\rho'$ are given by convergent power series in $(Z,\bar Z)\in \bC^{n+1}\times\bC^{n+1}$ and $(Z',\bar Z')\in \bC^{N+1}\times\bC^{N+1}$, respectively. By replacing $Z\bar Z$ and $\bar Z'$ by independent variables $\xi$ and $\bar \xi$, we obtain a holomorphic mapping $\mathcal H:=(H,\bar H)\colon U\times U^*\to \bC^{N+1}\times \bC^{N+1}$, where
$$
\bar H(\xi):=\overline{H(\bar \xi)},\quad U^*:=\{\xi\in \bC^{n+1}\colon \bar \xi\in U\},
$$
sending $0$ to $0$ and $\mathcal M$ into $\mathcal M'$, where $\mathcal M:=\{\rho(Z,\xi)=0\}\subset U\times U^*$ and $\mathcal M'=\{\rho'(Z',\xi')=0\} \subset \bC^{N+1}\times \bC^{N+1}$ denote the complexifications of $M$ and $M'$, respectively. Thus, we have
\begin{equation}\label{Hdefeq}
\rho'(H(Z),\bar H(\xi))=a(Z,\xi)\rho(Z,\xi),
\end{equation}
and $\mathcal H$ fails to be transversal to $\mathcal M'$ precisely along the common zero set of $a(Z,\xi)$ and $\rho(Z,\xi)$. If $M$ and $M'$ are merely $C^\infty$-smooth, then we can replace $\rho$, $a$, and $\rho'$ by their formal Taylor series at $0$ in $(Z,\bar Z)$ and $(Z',\bar Z')$ and $H$ by its convergent (or formal if $H$ is a $C^\infty$-smooth CR mapping) Taylor series at $0$ and obtain \eqref{Hdefeq} as an identity of formal power series. This is standard procedure in the field, and is referred to as identifying $M$ and $M'$ with their formal manifolds and considering $H$ as a formal mapping sending $M$ into $M'$; the reader is referred to e.g.\ \cite{BRgeom}, \cite{BER99a} or \cite{BER99b} for further discussion of this procedure. In what follows, we shall work over the rings of formal power series with formal manifolds and mappings, unless explicitly specified otherwise. For convenience, we shall also drop the $'$ on the target space coordinates $(Z',\xi')$, as it will be clear from the context to which space the variables belong.

It is well known, see e.g.\ \cite{BER99a}, that there are formal holomorphic (or convergent holomorphic in the real-analytic situation) normal coordinates $Z=(z,w)\in \bC^n\times \bC$ at $0\in M$ such that $M$ can be defined by a complex formal (again, convergent in the real-analytic case) equation
\begin{equation}\label{e:2.1}
 \rho(Z,\bar Z): = w - Q(z,\bar z,\bar w) = 0,
\end{equation}
where $Q(z,\xi)=Q(z,\chi,\tau)$ is a holomorphic function defined in a neighborhood of $0$ in $\C^n\times \C^n \times \C$ satisfying
\begin{equation}
 Q(z,0,\tau) \equiv Q(0,\chi,\tau) \equiv 0;
\end{equation}
The fact that equation \eqref{e:2.1} defines a real hypersurface is equivalent to one of the following the identities
\begin{equation}\label{reality}
 Q(z,\chi,\bar Q(\chi, z, w)) \equiv w \quad \text{or} \quad \bar Q(\chi,z,Q(z,\chi,\tau)) \equiv \tau;
\end{equation}
recall that if $u(x)$ is a formal power series in $x=(x_1,\ldots, x_q)$ (or holomorphic function), then we use the notation $\bar u(y) = \overline{u(\bar y)}$.
It follows that, in normal coordinates using $Z=(z,w)$ and $\xi=(\chi,\tau)$, we may use either of the choices
\begin{equation}\label{rhonc1}
\rho(Z,\xi)=w-Q(z,\chi,\tau)
\end{equation}
 or, in view of \eqref{reality},
 \begin{equation}\label{rhonc2}
 \rho(Z,\xi)=\tau-\bar Q(\chi,z,w).
 \end{equation}

We conclude this section by mentioning that in normal coordinates, the $T^{(1,0)}_0M$ space can be identified with the space of $c\in \bC^n\times \bC$ such that $c_{n+1}=0$ and the Levi form of $M$ at $0$ can be represented by the $n\times n$ matrix $Q_{z\chi}(0,0,0)$. Moreover, $M$ is of finite type at $0$ if and only if $Q(\chi,z,0)\not\equiv 0$.

\section{The non-transversality locus and proof of Theorem~\ref{main}}
In this section, we will assume $M$ and $M'$ are (analytic, smooth, or formal) real hypersurfaces in $\C^{n+1}$ and $\C^{N+1}$, respectively, and as mentioned in the previous section we shall assume that $p=0\in M$ and $p'=0\in M'$. We shall identify $M$ and $M'$ with formal hypersurfaces as explained in the previous section. We shall also assume in this section that
\begin{equation}\label{basiccodim}
2N-r \le 2n-2,
\end{equation}
where $r$ is the rank of Levi form of $M'$ at $p'=0$. We shall use the notation $\rho(Z,\xi)$ and $\rho'(Z',\xi')$ for (complexified) formal defining functions for $M$ and $M'$, respectively. Let $H\colon (\bC^{n+1},0)\to(\bC^{N+1},0)$ be a formal holomorphic mapping sending $M$ into $M'$, i.e.\
\begin{equation}\label{e:1}
\rho'(H(Z),\bar H(\xi)) = a(Z,\xi)  \rho(Z,\xi)
\end{equation}
where $a(Z,\xi)$ is a formal power series in $\C[[Z,\xi]]$. We shall assume that $H$ has generic rank $n+1$, i.e., there is at least one $(n+1)\times (n+1)$-minors of $H_Z$ which does not vanish identically. We shall denote by $\{\delta_l(Z), l=1,2,\dots, m\}$ the collection of $(n+1)\times (n+1)$-minors which do not vanish identically. (Thus, we have $1\leq m\leq \binom{N+1}{N-n}$.)

We first observe that $a\neq 0$ when \eqref{basiccodim} holds. For the readers convenience, we sketch the simple proof. Assume that $a\equiv 0$. Then, by differentiating \eqref{e:1} once with respect to $Z$ and once with respect to $\xi$, we obtain
\begin{equation}\label{firstLeviid}
H_Z^t(Z)\,\rho'_{Z\xi}(H(Z),\bar H(\xi))\,\bar H_\xi(\xi)=0,
\end{equation}
where as before $H_Z$ is the $(N+1)\times(n+1)$ Jacobian matrix of $H$; superscript $t$ denotes transpose of a matrix, and $\rho'_{Z\xi}$ is an $(N+1)\times (N+1)$-matrix. If we let $S$ denote the field of fractions of $\bC[[Z,\xi]]$, then we can regard \eqref{firstLeviid} as a matrix identity over $S$. Note that the ranks of $H^t_Z$ and $\bar H_\xi$ over $K$ are both $n+1$, and the rank of $\rho'_{Z\xi}$ is at least $r$ (since the rank of $\rho'_{Z\xi}(0,0)$ is at least $r$). Elementary linear algebra implies that $n+1-(N+1-r)\geq (N+1)-(n+1)$ or, equivalently, $2N-t\geq 2n$, proving our claim that $a\not\equiv 0$ under condition \eqref{basiccodim}. It now follows from Theorem 1.1 in \cite{BER07} that $a$ is not a multiple of $\rho$. In other words, in normal coordinates $Z=(z,w)$ and $\xi=(\chi,\tau)$ as in the previous section, we have $$
a(Z,(\chi,\bar Q(\chi,Z)))\not \equiv 0,\quad a((z,Q(z,\chi,\tau)),\xi)\not\equiv 0
$$
where $Q$ and $\bar Q$ are as in \eqref{rhonc1} and \eqref{rhonc2}, respectively.

As mentioned in the first section, the map $H$ is transversal to $M'$ at $0$ if and only if $a(0,0) \ne 0$. We shall consider the ideal $I:=I(a,\rho)$ in the ring $\bC[[Z,\xi]]$ of formal power series in $(Z,\xi)$. We that note the power series $a$ depends on the choices of defining power series $\rho$ and $\rho'$, but the ideal $I$ clearly does not. In the case $a(0,0) = 0$, the ideal $I$ is proper. When $M$, $M'$ and $H$ are analytic, $I$ defines a complex analytic variety $\mathcal X$ in $\C^{n+1}_Z \times \C^{n+1}_\xi$ consisting of the points at which the complexified map $\mathcal H(Z,\xi) = (H(Z),\bar H(\xi))$ fails to be transversal to the complexified hypersurface $\mathcal M'$. In this section, we shall also give a description (Corollary \ref{Xlocus}) of the non-transversality locus $\mathcal X$ (in the analytic case) of a (complexified) holomorphic map of generic full rank when the condition on the codimension of $W_H$ in Theorem \ref{main} fails; when the codimensional condition on $W_H$ holds, we shall show that $\mathcal X$ is empty.  But first let us observe some simple properties of~$I$.
\begin{lemma}\label{her} The ideal $I$ and its radical $\sqrt{I}$ are Hermitian, i.e.\ if $\alpha(Z,\xi) \in \C[[Z,\xi]]$, then $\alpha(Z,\xi) \in I$ if and only if $\bar \alpha(\xi,Z) \in I$, and similarly for~$\sqrt{I}$.
\end{lemma}
Recall that the radical of $I$ is defined by $\sqrt I = \{f: f^q \in I \text{ for some integer } q\}$.
\begin{proof} Recall that we can choose real-valued defining functions $\rho$ and $\rho'$ for $M$ and $M'$ respectively and, hence, the corresponding function $a$ is real-valued as well. At the level of formal power series, this is equivalent to $\rho$, $\rho'$, $a$ being Hermitian; i.e. if $u$ equals $\rho$, $\rho'$, or $a$, then $u(Z,\xi) = \bar u(\xi, Z)$. The conclusion of Lemma \ref{her} follows immediately.
\end{proof}
In the following lemma, we use normal coordinates $Z=(z,w)$, $\xi=(\chi,\tau)$ as above, and $\bar Q(\chi,Z)=\bar Q(\chi,z,w)$ is the power series appearing in \eqref{rhonc2}. Recall that $a(Z,(\chi,\bar Q(\chi,Z)))\not\equiv 0$.
\begin{lemma}\label{primedecomposition} Assume that $a(0,0)=0$ and let
\[
a(Z,(\chi,\bar Q(\chi,Z))) = a_1^{t_1}(Z,\chi) \dots a_k^{t_k}(Z,\chi)
\]
be the unique (modulo units) factorization into irreducible (or prime) elements in $\bC^[[Z,\chi]]\subset \C[[Z,\xi]]$. Let $I_j = I(a_j,\rho)$. Then, $$\sqrt I = \cap_{j=1}^k I_j$$ is a prime decomposition of $\sqrt I$.
\end{lemma}

\begin{proof} Recall that we may choose $\rho(Z,\xi) = \tau - \bar Q(\chi,Z)$. It then follows that for some $\tilde a(Z,\xi)\in \C[[Z,\xi]]$
\[
a(Z,\xi) = a(Z,(\chi,\bar Q(\chi,Z))) + \tilde a(Z,\xi) \rho(Z,\xi).
\]
Hence $a\in I_j$ and so $I = I(a,\rho) \subset I_j$ for all $j = 1,\dots k$.

Next, we claim that, for each $j$, the ideal $I_j$ is prime. Indeed, fix $j$ and let $f,g \in \C[[Z,\xi]]$ such that $fg \in I_j$. Then
\[
f(Z,\xi)  g(Z,\xi) = r(Z,\xi) a_j (Z,\chi) + s(Z,\xi)  \rho(Z,\xi),
\]
for some $r,s\in \C[[Z,\xi]]$. If we substitute $\xi = (\chi,\bar Q(\chi,Z))$ in this identity, then we obtain
\[
f(Z,(\chi,\bar Q(\chi,Z))) g(Z,(\chi,\bar Q(\chi,Z))) = r(Z,(\chi,\bar Q(\chi,Z))) a_j (Z,\chi)
\]
Since $a_j(Z,\chi)$ is irreducible, we deduce that it divides, say, $f(Z,(\chi,\bar Q(\chi,Z)))$. It follows that
\begin{equation}
\begin{aligned}
f(Z,\xi) &= f(Z,(\chi,\bar Q(\chi,Z))) + \tilde f(Z,\xi)  \rho(Z,\xi)\\
&= r(Z,(\chi,\bar Q(\chi,Z))) a_j (Z,\chi)+ \tilde f(Z,\xi)  \rho(Z,\xi)
\end{aligned}
\end{equation}
for some $\tilde f(Z,\xi)$ and so $f$ belongs to $I_j$. We conclude that $I_j$ is prime, as desired. Since $I\subset I_j$, for all $j$, and $I_j$ is prime, we conclude that $\sqrt I\subset I_j$, for all $j$, proving $\sqrt I\subset \cap_{j=1}^kI_j$.

Now assume $f(Z,\xi) \in I_j$ for all $j$. Then we can write, for any fixed $j$,
\[
f(Z,\xi) = r(Z,\xi) a_j(Z,\chi) + s(Z,\xi)  \rho(Z,\xi),
\]
for some power series $r$ and $s$.
If we substitute $\tau = \bar Q(\chi,Z)$, then we get
\[
f(Z,(\chi, \bar Q(\chi,Z))) = r(Z, (\chi, \bar Q(\chi,Z)))  a_j(Z,\chi).
\]
Thus $f(Z,(\chi, \bar Q(\chi,Z))) $ is divisible by $a_j(Z,\chi)$ for all $j = 1,2,\dots k$. It then follows that, for some integer $l$, $f(Z,(\chi, \bar Q(\chi,Z)))^l$ is divisible by $a(Z,(\chi,\bar Q(\chi,Z)))$. We conclude that
\[
f(Z,\xi)^l= f(Z,(\chi, \bar Q(\chi,Z)))^l + \tilde f(Z,\xi) \rho(Z,\xi),
\]
for some $\tilde f(Z,\chi)$, belongs to $I$ and hence $f(Z,\xi) \in \sqrt I$. Consequently, $\cap_{j=1}^k I_j\subset \sqrt I$. The proof is complete.
\end{proof}
A key point in the proof of our main results is the following lemma.
\begin{lemma}\label{keylemma} Assume that $a(0,0) = 0$. Then, for each $j$,  either $\delta_l(Z) \in I_j$ for every $l=1,2,\dots m$, or $\bar \delta_l(\xi) \in I_j$ for every $l=1,2,\dots m$. As above, $\delta_l(Z), l=1,2,\dots m$ denotes the collection of all $(n+1) \times (n+1)$-minors of $H_Z(Z)$ that do not vanish identically.
\end{lemma}
\begin{proof} If we differentiate \eqref{e:1} with respect to $Z$ we obtain
\begin{equation}\label{e:2}
 H_Z(Z)^t \, \rho'_Z(H(Z),\bar H(\xi)) = a_Z(Z,\xi)\rho(Z,\xi) + a(Z,\xi) \rho_Z(Z,\xi).
\end{equation}
Here, as above, the Jacobian matrix $H_Z$ is regarded as an $(N+1)\times (n+1)$-matrix, the superscript $t$ denotes transposition of a matrix, and the gradient vectors are regarded as column vectors. Let $K=\C[[Z]]$. Then $\C[[Z,\xi]]$ can be identified with the ring $K[[\xi]]$.
We can regard equation \eqref{e:2} as an identity in $(K[[\xi]])^{N+1}$. Thus, we may rewrite this identity as follows
\begin{equation}\label{e:3}
H_Z^t \, \rho'_Z(\bar H(\xi)) = a_Z(\xi)\rho(\xi) + a(\xi) \rho_Z(\xi),
\end{equation}
where we have used the notation $\rho'_Z(\xi') := \rho'_Z(H(Z),\xi')$; $H^t_Z$ is a matrix with components in the field $K$ and e.g.\ $a(\xi)$ is a formal power series in $\xi$ whose coefficients are elements in $K$.

Since $I_j$ is proper prime ideal of $K[[\xi]]$, it follows that $K[[\xi]]/I_j$ is an integral domain. Let us fix a $j$, and define $S$ to be the field of
fractions of $K[[\xi]]/ I_j$. Denote by $\pi$ the canonical projection: $\pi\colon K[[\xi]] \to K[[\xi]]/I_j,\ x \mapsto x + I_j$.

Now, let $L$ be a formal vector field (or a derivation) in $K[[\xi]]$, i.e.\
\[
 L = \sum_{l=1}^{n+1} \beta_l(\xi) \frac{\partial}{\partial \xi_l}
\]
where $\beta_l(\xi) \in K[[\xi]]$. We say that $L$ is Zariski tangent to $I_j$ if $L(f) = 0 \mod I_j$ for all $f\in I_j$, or equivalently,
\begin{equation}\label{e:4}
 \sum_{l=1}^{n+1} \beta_l(\xi) \frac{\partial a_j}{\partial \xi_l} (\xi)=\sum_{l=1}^{n+1} \beta_l(\xi) \frac{\partial \rho}{\partial \xi_l}(\xi) = 0 \mod I_j.
\end{equation}
It is straightforward to see that there are at least $n-1$ formal vector fields $L_1,\dots L_{n-1}$ tangent to~$I_j$,
\begin{equation}\label{Lk}
L_k = \sum_l \beta^{k}_l(\xi)\partial/ \partial \xi_l,
\end{equation}
such that
the collection of corresponding vectors in $S^{n+1}$:
$$\hat V_k = (\pi(\beta^{k}_1(\xi)), \dots, \pi(\beta^{k}_{n+1}(\xi))),\quad k = 1,2,\dots n-1,$$
is linearly independent over the quotient field $S$ of $K[[\xi]]/I_j$. Indeed, let us consider the following system of two
linear equations over $S$ with unknowns $X_l$, $l=1,2,\dots n+1$,
\[\sum_{l=1}^{n+1}X_l\,  \pi(a_{j,\xi_l}(\xi)) = 0, \quad \sum_{l=1}^{n+1}X_l \, \pi(\rho_{\xi_l}(\xi)) = 0.\]
This system has at least $n-1$ linearly independent solutions in $S^{n+1}$, denoted by
$$\tilde V_k=(\tilde\beta^k_1,\ldots,\tilde\beta^k_{n+1}),\quad k=1,\ldots,n-1.$$
where $\tilde \beta^k_l\in S$. Since each component $\tilde\beta^k_l$ is a fraction $\tilde\beta^k_l=\mu^k_l/\nu^k_l$ with $\mu^k_l,\nu^k_l\in K[[\xi]]/I_j$, we can clear the denominators and obtain $n-1$ linearly independent vectors $$\hat V_k=(\hat\beta^k_1,\ldots,\hat\beta^k_{n+1}),$$ whose components belong to $K[[\xi]]/I_j$. Since $\pi \colon K[[\xi]] \to K[[\xi]]/I_j$ is surjective, we can find $\beta^k_l(\xi)\in K[[\xi]]$ such that $\pi(\beta^k_l(\xi))=\hat\beta^k_l$. The corresponding formal vector fields $L_k$, given by \eqref{Lk}, satisfy the desired properties.

We now apply the vector fields $L_k$ to \eqref{e:3}. It follows from \eqref{e:4} and the fact that $I\subset I_j$ (Lemma \ref{primedecomposition}) that
\[
 L_k\big(a_Z(\xi) \rho(\xi) + a(\xi) \rho_Z(\xi)\big) = 0 \mod I_j \ \text{for} \ k =1,2,\dots n-1.
\]
Consequently, we have the following identity
\begin{equation}\label{e:5}
 L_k\big(H_Z^t \, \rho'_Z(\bar H(\xi))\big) = 0 \mod I_j  \ \text{for} \ k =1,2,\dots n-1.
\end{equation}
Using chain rule, we can rewrite \eqref{e:5} in matrix notations as follows
\begin{equation}\label{e:6}
H_Z^t\,   \Phi(\xi)\,  \bar H_{\xi}(\xi)\,  V_k(\xi) = 0 \mod I_j \ \text{for} \ k =1,2,\dots n-1,
\end{equation}
where we have used the notation $V_k(\xi) =(\beta_1^{k}(\xi),\dots, \beta_{n+1}^{k}(\xi))^t$ and $\Phi(\xi)$ for the $(N+1)\times(N+1)$ matrix $(\rho'_{Z\xi}(\bar H(\xi)))$ . Note that since the Levi form of $M'$ at $0$ (which is represented by the restriction of $\rho'_{Z\xi}(0,0)$ to the holomorphic tangent space of $M'$ at $0$) has rank $r$ by assumption, there is an $r\times r$-minor of $\Phi$ which is a unit in $K[[\xi]]$.

We now go to the quotient field $S$ of $K[[\xi]]/I_j$. We will put a hat over elements of $K[[\xi]]$ (including vectors and matrices with elements in $K[[\xi]]$) to indicate
 their images in $K[[\xi]]/I_j$ under the canonical projection $\pi$. Thus, \eqref{e:6} implies
\begin{equation}\label{e:7}
 \hat H_Z^t \, \hat \Phi \, \hat{\bar{H_{\xi}}} \, \hat V_k = 0  \ \text{for} \ k =1,2,\dots n-1.
\end{equation}
Let us assume, in order to reach a contradiction, that there is at least one $\delta_l(Z)\not\in I_j$ and at least one $\bar \delta_{l'}(\xi) \not\in I_j$. Consequently, the corresponding minors $\hat\delta_l$ and $\hat{\bar \delta}_{l'}$ do not vanish in $S$ and it follows that the matrices $\hat H_Z^t $ and $\hat{\bar{H_{\xi}}}$ have rank $n+1$ over the field $S$. Furthermore, since there is an $r\times r$-minor of the matrix $\Phi$
which is a unit in $K[[\xi]]$, we deduce that $\hat \Phi$ has rank at least $r$ over $S$.
Now, consider the following collection of vectors in $S^{N+1}$:
\begin{equation}\label{e:6.1}
 Y_k = \hat \Phi \ \hat{\bar{H_{\xi}}} \ \hat V_k, \quad k=1,\dots n-1 \quad \text{and} \quad Y_n = (\hat {\rho'}_{z_1}, \dots \hat{ \rho'}_{z_{N+1}}).
\end{equation}
Here we recall that $\hat {\rho'}_{z_l} = \pi(\rho'_{z_l}(\bar H(\xi)))$.
We claim that the rank of the collection of vectors $Y_1,\ldots Y_n\in S^{N+1}$ is at least $n+r-N-1$. Indeed, since $\hat\Phi$ has rank at least $r$ over $S$ and the $(N+1)\times(n+1)$-matrix $\hat{\bar H}_\xi$
has full rank ($=n+1)$, we deduce that the collection $Y_1,\ldots Y_{n+1}$ has rank at least $n+r-N-2$. On the other hand, observe that in normal coordinates in the target space we may choose $\rho'(Z,\xi) = w - Q'(z,\chi,\tau)$ and, hence, the last row
of $\hat \Phi$ contains only zeros resulting in the last component of each $Y_k$, for
$k=1,\dots n-1$, being 0. On the other hand, the last component of $Y_n$ is $1$, so
that $Y_n$ cannot be a linear combination of $Y_k$ for $k=1,\dots, n-1$. The claim that the rank of the collection $\{Y_k\}_{k=1}^{n}$ is $n+r-N-1$ follows. To complete the proof of the lemma, we observe from \eqref{e:3} and \eqref{e:7} that $ \hat H_Z^t \, Y_k = 0$ for $k=1,\dots n$. Since the $(n+1)\times(N+1)$-matrix $\hat H_Z^t$ has rank $n+1$, we deduce that $(N+1) - (n+1) \ge n+r-N-1$. This implies $2N-r \ge 2n-1$ which contradicts \eqref{basiccodim}. The proof of Lemma \ref{keylemma} is complete.
\end{proof}

\begin{proposition}\label{ridiculous} Suppose that $a(0,0)=0$. Then there are $B(Z), C(Z) \in \C[[Z]]$ such that
\begin{equation}\label{aform}
a(Z,\xi) = B(Z) \bar C(\xi)t(Z,\xi) + \tilde a(Z,\xi) \rho(Z,\xi),
\end{equation}
for some $\tilde a(Z,\xi), t(Z,\xi)\in \bC[[Z,\xi]]$ such that $t(Z,\xi)$ is a unit.
Moreover, each irreducible divisor of $B(Z)$ and each irreducible divisor of $C(Z)$ divides $\delta_l(Z)$ for every $l=1,\ldots,m$.
\end{proposition}
\begin{proof} By Lemma~\ref{keylemma}, for each $j$, we have either $\delta_l(Z) \in I_j$ for all $l=1,\ldots,m$ or $\bar \delta_j(\xi) \in I_j$ for all $l$.
Let us first assume that $\delta_l(Z) \in I_j$ for all $l$. Consequently,
\[
\delta_l(Z) =  r_l(Z,\xi)\ a_j(Z,\chi) + s_l(Z,\xi)\ \rho(Z,\xi)
\]
for some $r_l$ and $s_l$.
By substituting $\tau=\bar Q(\chi,Z)$ we obtain
\[
\delta_l(Z) =  r_l(Z,(\chi,\bar Q(\chi,Z))) \ a_j(Z,\chi)
\]
Now, recalling that $a_j(Z,\chi)$ is irreducible, we conclude that
\begin{equation}
a_j(Z,\chi) = b_j(Z)\ u_l(Z,\chi)
\end{equation}
where $b_j(Z)$ is an irreducible (or prime) divisor of $\delta_l(Z)$ and $u_l(Z,\chi)$ is an unit. This holds for all $l=1,2,\dots m$ (with the same $b_j(Z)$, modulo units) and hence $b_j(Z)$ is an irreducible divisor of $\delta_l(Z)$ for every $l=1,\ldots,m$.

If $\bar \delta_l(\xi) \in I_j$ for all $l$, then we can write $\bar \delta_l(\xi) = r_l(Z,\xi) \ a_j(Z,\chi) + s_l(Z,\xi) \ \rho(Z,\xi)$. Thus, by substituting $w = Q(Z,\xi)$ this time, we obtain
\[
 \bar \delta_l(\xi) = r_l(Z,(\chi, \bar Q(\chi,Z))) \ a_j((z,Q(z,\xi)),\chi)
\]
Observe that $a_j((z,Q(z,\xi)),\chi)$ is also irreducible, a fact that follows easily from the identity
$$
a_j(Z,\chi)=a_j((z,w),\xi)=a_j((z,Q(z,\chi,\bar Q(\chi,z,w))),\chi),
 $$
 where we recall that $\xi=(\chi,\tau)$.
 It then follows as above that $a_j((z,Q(z,\xi)),\chi)$ is an irreducible divisor of $\bar \delta_l(\xi)$ and thus $a_j((z,Q(z,\xi)),\chi) = \bar c_j(\xi) \ v_l(Z,\xi)$
for some irreducible divisor $\bar c_j(\xi)$ of $\bar\delta_l(\xi)$ and unit $v_l(Z,\xi)$. Again, since $\bar\delta_l(\xi)\in I_j$ for every $l$, we conclude that $\bar c_j(\xi)$ divides $\bar\delta_l(\xi)$ for every $l$. If we substitute $\tau = \bar Q(\chi, Z)$, then we obtain
\[
 a_j(Z,\chi) = \bar c_j(\chi,\bar Q(\chi,Z)) \ \tilde v_l(Z,\chi),
\]
where $\tilde v_l(Z,\chi)=v(Z,(\chi,\bar Q(\chi,Z))$.
Putting all this together, we conclude (via Lemma \ref{primedecomposition}) that
\begin{align}
a(Z,(\chi,\bar Q(\chi,Z))) = a_1^{t_1}(Z,\chi) \cdots a_k^{t_k}(Z,\chi) = B(Z) \bar C(\chi,\bar Q(\chi, Z)) \ t(Z,\xi)
\end{align}
where $t(Z,\xi)$ is a unit.
By construction, every irreducible divisor of $B(Z)$ divides  $\delta_l(Z)$ for all $l$, and every irreducible divisor of $\bar C(\xi)$ divides $\bar \delta_l(\xi)$ for all $l$. We conclude that
\begin{align}\label{apple}
a(Z,\xi) &= a(Z,(\chi,\bar Q(\chi,Z)))  + \tilde a(Z,\xi)  \rho(Z,\xi)\\
&=B(Z)\ C(\chi,\bar Q(\chi, Z)) \ t(Z,\xi) + \tilde a(Z,\xi)  \rho(Z,\xi) \notag.
\end{align}
Similarly, we can also write
\[
 \bar C(\chi, \bar Q(\chi, Z))= \bar C(\xi)  + \tilde C(Z,\xi) \rho(Z,\xi),
\]
which by substituting into \eqref{apple} yields the desired form of $a(Z,\xi)$. The proof is complete.
\end{proof}

We may now prove the following result, which as explained above is a reformulation of Theorem \ref{main} in the formal setting (and hence has Theorem \ref{main} as a direct consequence).

\begin{theorem}\label{funny} Let $M$ and $M'$ be formal real hypersurfaces through $0$ in $\C^{n+1}$ and $\C^{N+1}$, respectively, and $H\colon (\bC^{n+1},0)\to (\bC^{N+1},0)$ a formal holomorphic mapping sending $M$ into $M'$. Assume that $$2N-r \le 2n-2,$$ where $r$ is the rank of Levi form of $M'$ at $0$. Assume further that the Jacobian matrix $H_Z$ is of generic rank $n+1$ (i.e.\ at least one $(n+1)\times(n+1)$-minor is not identically zero) and that the collection of its not-identically-zero $(n+1)\times (n+1)$-minors $\delta_1,\ldots,\delta_m$ have no nontrivial common divisor. Then, $H$ is transversal to $M'$ at $0$.
\end{theorem}
\begin{proof}
 Assume, in order to reach a contradiction, that $H$ is not transversal to $M'$ at $0$, i.e.\ $a(0,0)=0$ where $a(Z,\xi)$ is given by \eqref{e:1}. By Proposition~\ref{ridiculous}, there are nontrivial power series $B(Z)$ and $C(Z)$ such that \eqref{aform} holds, and such that every irreducible divisor of $B(Z)$ and every irreducible divisor of $C(Z)$ divides $\delta_l(Z)$ for all $l$. Also, note that at least one of $B(Z)$ or $C(Z)$ has to be 0 at $Z=0$, since $a(0,0)=\rho(0,0)=0$. This contradicts the assumption that $\delta_1(Z),\ldots, \delta_m(Z)$ have no common divisor. The proof is complete.
\end{proof}

We conclude this section by giving a description in the analytic case (i.e.\ $M$, $M'$, and $H$ are analytic) of the non-transversality locus $$\mathcal X:=\{(Z,\xi)\in \bC^{n+1}\times\bC^{n+1}\colon a(Z,\xi)=\rho(Z,\xi)=0\}=\{(Z,\xi)\in \mathcal M\colon a(Z,\xi)=0\}$$ of the complexified map $\mathcal H(Z,\xi)=(H(Z),\bar H(\xi))$  when \eqref{basiccodim} holds but the codimension of $W_H$ is one. (Of course, when the codimension of $W_H$ is at least two, we just proved that $\mathcal X$ is empty.) Recall (see e.g.\ Example \ref{ex1}) that $\mathcal X$ may be non-empty in this situation, but it turns out that the variety must have a special form.  In context of Segre preserving maps, a similar description was given in \cite{Zha}.

\begin{corollary}\label{Xlocus}
Let $M$ and $M'$ be real-analytic hypersurfaces through $0$ in $\C^{n+1}$ and $\C^{N+1}$, respectively, and $H\colon (\bC^{n+1},0)\to (\bC^{N+1},0)$ a holomorphic mapping sending $M$ into $M'$. Assume that $$2N-r \le 2n-2,$$ where $r$ is the rank of Levi form of $M'$ at $0$, and that the Jacobian matrix $H_Z$ is of generic rank $n+1$ (i.e.$W_H$ is a proper subvariety). If $H$ is not transversal to $M'$ at $0$ then the non-transversality locus $\mathcal X = \{(Z,\xi)\in \mathcal M : a(Z,\xi) = 0\}$ of the complexified map $\mathcal H=(H,\bar H)$ has a decomposition into irreducible components of the following form
\[
 \mathcal X = \mathcal X_1 \cup \dots \cup \mathcal X_k,
\]
where $\mathcal X_j$ is either of the form $\{(Z,\xi): Z \in W_i, \xi \in S^*_Z\}$ or $\{(Z,\xi): \xi \in W_i^*, Z \in S_{\bar\xi}\}$. Here, the $W_i$ denote the irreducible, codimension one components of $W_H$, $^*$ denotes the complex conjugate of a set, and
$$
S_p:=\{Z\in \bC^{n+1}\colon \rho (Z,\bar p)=0\}
$$ denotes the Segre varieties of $M$ at $p$. Moreover, $\mathcal X$ is Hermitian symmetric, i.e.\ $\mathcal X^*=\mathcal X$.
\end{corollary}
\begin{proof}
 Observe that by Proposition~\ref{ridiculous}, $(Z,\xi) \in \mathcal X$ if and only if $B(Z)\bar C(\xi) = 0$, where each irreducible factor of $B$ and $C$ divide every $(n+1)\times (n+1)$-minor $\delta_l$ of $H$. The decomposition of $\mathcal X$ in Corollary \ref{Xlocus} follows readily from this fact. The Hermitian symmetry is immediate from Lemma \ref{her}. The proof is complete.
\end{proof}

\section{Proof of Theorem~\ref{mains}}

In this section, we will assume $M$ and $M'$ are (analytic, smooth or formal) real hypersurfaces in $\C^{n+1}$ and $\C^{N+1}$, respectively, and that $M$ is of finite type at $0$. The aim is to show that $H$ is transversal to $M'$ at $0$ under the assumptions in Theorem \ref{mains}. In view of Theorem \ref{funny} (since condition \eqref{codims} implies $2N-r\leq 2n-2$ for any $s\leq n+1$), we may therefore assume here that the non-empty collection of non-trivial $(n+1)\times (n+1)$-minors of $H_Z(Z)$, denoted as before by $\delta_l(Z)$ for $l=1,\ldots,m$, has a (non-trivial) common divisor. The main result of this section is the following theorem.

\begin{theorem}\label{boring} Let $M\subset \C^{n+1}$ and $M\subset \C^{N+1}$ be formal real hypersurface and $H$ formal map such that $H(M)\subset M'$ and the collection of not identically zero $(n+1)\times (n+1)$-minors of $H_Z(Z)$, denoted as before by $\delta_l(Z)$ for $l=1,\ldots,m$, is non-empty. Let $r$ be the rank of the Levi form of $M'$ at $0$ and assume that, for some $1\leq s\leq n+1$,
\begin{equation}\label{codim3}
 2N-r \le n+s-3.
\end{equation}
Suppose further that $M$ is of finite type at $0$, and that for every common divisor $d(Z)$ of the collection $\delta_l(Z)$, $l=1,\ldots,m$, there is at least one $k\times k$-minor $\delta'(Z)$ of the Jacobian matrix $H_Z(Z)$ such that $k\geq s$ and $\delta'(Z)$ is relatively prime to $d(Z)$. Then $H$ is transversal to $M'$ at $0$.
\end{theorem}
For the proof, we will need the following lemma.
\begin{lemma}\label{stupid} Assume that $M$ is of finite type at $0$ and $H$ is not transversal to $M'$, i.e., $a(0)=0$. Let $I_j$ be the ideal defined in Lemma $\ref{primedecomposition}$. If there exists a non-trivial $\alpha(Z) \in I_j$ then there are no non-trivial $\beta(\xi) \in I_j$.
\end{lemma}
\begin{proof} We assume, in order to obtain a contradiction, that there are
non-trivial power series $\alpha(Z),\beta(\xi) \in I_j = I(a_j,\rho)$. We can argue as in the proof of Proposition~\ref{ridiculous} to deduce that
\begin{equation}\label{A}
 a_j(Z,\chi) = b(Z) \ u(Z,\chi).
\end{equation}
where $u(Z,\chi)$ is a unit and $b(Z)$ an irreducible divisor of $\alpha(Z)$ in $\C[[Z]]$.
Similarly, we can deduce from the fact that $\beta(\xi) \in I$ that
\begin{equation}\label{B}
 a_j(Z,\chi) = c(\chi,\bar Q(\chi,Z)) \ v(Z,\xi)
\end{equation}
where $v\in \C[[z,\xi]]\subset \C[[Z,\xi]]$ is a unit and $c(\xi)$ is a divisor of $\beta(\xi)$ in $\C[[\xi]]$.
Now, we deduce from \eqref{A} and \eqref{B} that
\begin{align*}
u(Z,\xi) \ b(Z) = a_j(Z,\chi) = v(Z,\xi) \ c(\chi,\bar Q(\chi,Z)).
\end{align*}
Hence, for some unit $s(Z,\xi)$,
\begin{equation}\label{e:10}
  b(Z) = c(\chi,\bar Q(\chi,Z)) \ s(Z,\xi).
\end{equation}
If we substitute $Z=0$ into \eqref{e:10}, then we get
\begin{equation*}
c(\chi,0) = c(\chi,\bar Q(\chi,0)) = s(0,\xi)^{-1} \ b(0) = 0.
\end{equation*}
We deduce that $c(\xi) = \tau  \tilde c(\xi)$, where $\tilde c(\xi)$ is an unit since $c(\xi)$ is irreducible. Then, by setting $\chi = 0$ in \eqref{e:10} and recalling that $\bar Q(0,Z) = w$, we deduce that for some unit $\tilde b(Z,\xi)$:
\begin{align}\label{e:11}
 b(Z) =  c(0,\bar Q(0,Z)) \ s(Z,0,\tau) = w \ \tilde b(Z,\xi).
\end{align}
Consequently:
\begin{equation}\label{e:12}
w \ \tilde b(Z,\xi) = b(Z) = c(\chi,\bar Q(\chi,Z)) \ s(Z,\xi) = \bar Q(\chi,Z)\ \tilde c(\chi,\bar Q(\chi,Z)\ s(Z,\xi).
\end{equation}
This and the fact that $s$ and $\tilde c$ are units imply $\bar Q(\chi, z, 0) = 0$. This contradicts the fact that $M$ is of finite type at $0$. The proof is complete.
\end{proof}
Now we can prove the Theorem~\ref{boring}.
\begin{proof}[Proof of Theorem~$\ref{boring}$] We will argue by contradiction. Assume that $H$ is not transversal to $M'$ at $0$, i.e.\ $a(0,0)=0$. Recall that $I_j$, $j=1,\ldots, k$, denote the ideals defined in Lemma \ref{primedecomposition}.
By Lemma~\ref{keylemma}, for each $j$ either $\delta_l(Z) \in I_j$ for all $l$ or $\bar \delta_l(\xi) \in I_j$ for all $l$. We claim that for some $j$, $\delta_l(Z)\in I_j$ for all $l$. Indeed, even if $\bar \delta_l(\xi)\in I_j$ for all $l$ and all $j$, then (by Lemma \ref{primedecomposition}) $\bar\delta_l(\xi)\in\sqrt{I}$ and, hence, by Lemmas \ref{her} and \ref{primedecomposition} we would also have $\delta_l(Z)\in I_j$ for all $l$ and $j$. (Although this does not matter for the proof, we point out that this latter situation cannot occur by Lemma \ref{stupid}.)

Let us now fix a $j$ be such that $\delta_l(Z)\in I_j$ for all $l$. We claim that there is an $k\times k$-minor $\delta'(Z)$ of $H_Z(Z)$ such that $k\geq s$ and $\delta'(Z) \not\in I_j$. Indeed, if $\delta_i'(Z)$, for $i=1,\ldots,p$, denote the collection of all $k\times k$-minors of $H_Z(Z)$ for $k\geq s$ and $\delta'_i(Z)\in I_j$ for all $i$, then we can argue as in the proof of Proposition~\ref{ridiculous} (considering the collection of $\delta'_i(Z)$ and $\delta_l(Z)$, for $i=1,\ldots, p$ and $l=1\ldots, m$) and conclude that there is a common irreducible divisor $b(Z)$ of $\delta'_i(Z)$ and $\delta_l(Z)$ for all $i=1,\ldots, p$ and $l=1,\dots m$ (which is also a divisor of $a_j(Z,\chi)$, although this does not matter here). This contradicts the fact that to every common divisor of the $\delta_l(Z)$, $l=1,\ldots, m$, there is at least one $\delta'_i(Z)$ which is relatively prime to it. Thus, let $\delta'(Z)$ denote a $k\times k$-minor, with $k\geq s$, such that $\delta'(Z) \not\in I_j$. We will now proceed along the lines of the proof of Lemma~\ref{keylemma}, using the same notation as in that proof. We first observe that, by Lemma \ref{stupid}, no $\bar \delta_l(\xi)\in I_j$ by our choice of $j$. We conclude that the rank of $\hat{\bar H}_\xi(\xi)$ over the quotient field $S$ of $K[[\xi]]/I_j$ is $n+1$. On the other hand, since $\delta'(Z) \not\in I_j$, it also follows that the rank of $\hat H_Z(Z)$ over $S$ is at least $k\geq s$. We then argue in the same way as in the proof of Lemma \ref{keylemma} to deduce from \eqref{e:7} that $2N-r \ge n+s-2$, which contradicts our assumption \eqref{codim3}. The proof is complete.
\end{proof}

\section {Proof of claim in Example \ref{ex0.5}}\label{example}

We first note that
\begin{equation}\label{exeq1}
\begin{aligned}
\left |z_j+[z]\, z_j+\frac{i}{2}w\right |^2-\left |z_j-[z]\, z_j-\frac{i}{2}w\right |^2 &=  2\left(z_j\left(\overline{[z]z_j}-\frac{i}{2}\bar w\right)+\bar z_j\left([z] z_j+\frac{i}{2}w\right)\right)\\
& = 2\left (\left([z]+\overline{[z]}\right)|z_j|^2 +\frac{i}{2}(\bar z_jw -z_j\bar w)\right)
\end{aligned}
\end{equation}
Thus, it follows that the expression
\begin{equation}
2i\left (\im\, G +\sum_{j=1}^n |F_{2j-1}|^2-\sum_{j=1}^n |F_{2j}|^2\right ) = G-\bar G+2i\sum_{j=1}^n\left (|F_{2j-1}|^2-|F_{2j}|^2\right)
\end{equation}
is equal to
\begin{multline}
-2\left([z]w-\overline{[z]w} -2i\left([z]+\overline{[z]}\right)\sum_{j=1}^n |z_j|^2 +\overline{[z]}w-[z]\bar w\right)\\=
-2\left([z]+\overline{[z]}\right)(w-\bar w-2i\sum_{j=1}^n |z_j|^2),
\end{multline}
or, in other words,
\begin{equation}
\im\, G +\sum_{j=1}^n |F_{2j-1}|^2-\sum_{j=1}^n |F_{2j}|^2 = i\left([z]+\overline{[z]}\right)(w-\bar w-2i\sum_{j=1}^n |z_j|^2),
\end{equation}
proving that $H$ sends $M$ into $M'$, and that $H$ is not transversal to $M'$ along the intersection of $M$ with $\re\, [z] =0$ as claimed in Example \ref{ex0.5}. The fact that $H$ is a local embedding at $0$ is trivial.

\end{document}